\documentclass[12pt,reqno]{amsart}
\usepackage{amssymb,amsmath,amsthm,newlfont,enumerate}
\newcommand{\card}{\mathrm{card}\,}
\newtheorem{theorem}{Theorem}

\newtheorem{proposition}[theorem]{Proposition}
\newcommand{\T}{\mathbb{T}}

\newcommand{\D}{\mathbb{D}}

\DeclareMathOperator{\clos}{clos}
\newcommand{\WEP}{{\text{\rm WEP}}}
\newcommand{\ZT}{{Z
(\Theta)}}
\newcommand{\ZB}{{Z
(B)}}

\begin{document}
\title[Generalized Carleson--Newman inner functions]{Generalized Carleson--Newman inner functions}
\author[Alexander Borichev]{Alexander Borichev}
\address{LATP\\CMI\\Aix--Marseille Universit\'e\\
39, rue F. Joliot-Curie\\13453 Marseille\\France}
\email{borichev@cmi.univ-mrs.fr}

\begin{abstract} We study a class of inner functions introduced by 
Gorkin, Mortini, and Nikolski, and motivated by Banach algebras and functional calculus applications. Answering their question, we produce a singular function that cannot be multiplied into this generalized Carleson--Newman 
class of inner functions (the class of functions satisfying the so called 
weak embedding property). Furthermore, we give elementary proofs for some results on these classes obtained earlier using arguments related to 
the Gelfand representation of $H^\infty$.
\end{abstract}
\thanks{Research partially supported by the ANR grant FRAB}

\maketitle

'\section{Introduction}

Given an inner function $\Theta$ in the unit disc $\D$, denote by 
$\ZT$ its zero set in $\D$. The pseudo-hyperbolic distance 
in the unit disc is defined as  
$$
\rho(z,w)=\frac{|z-w|}{|1-\bar z w|}.
$$
V.~Vasyunin proved in \cite{VA} that a Blaschke product $B$ is interpolating if and only if for some $c>0$ we have
$$
|B(\lambda)|\ge c\,\rho(\lambda,\ZB),\qquad \lambda\in\D,
$$
where 
$$
\rho(\lambda,\ZT)=\inf\bigl\{\rho(\lambda,w):w\in \ZT\bigr\}.
$$
For interpolating Blaschke products and other notions related to the function theory in the unit disc
see, for instance, \cite{G}.

Furthermore, P.~Gorkin and R.~Mortini \cite{GM} proved that a Blaschke 
product $B$ is a product of $N$ interpolating Blaschke products 
(is a Carleson--Newman Blaschke product) if and only if 
for some $c>0$ we have
$$
|B(\lambda)|\ge c\,\rho(\lambda,\ZB)^N,\qquad \lambda\in\D.
$$

Recently, P.~Gorkin, R.~Mortini, and N.~Nikolski \cite{GMN} 
introduced an interesting class of inner functions determined by an asymptotic condition on its level sets. Let us say that an inner function 
$\Theta$ 
is {\it a generalized Carleson--Newman inner function}
($\Theta$ satisfies the so called weak embedding property ($\WEP$), a 
weakening of the well-known Carleson embedding property)
if and only if for every $\varepsilon>0$ there exists $\delta>0$ 
such that
$$
|\Theta(\lambda)|<\delta \implies \rho(\lambda,\ZT)<
\varepsilon,\qquad \lambda\in\D.
$$

It is proved in \cite{GMN} that $\Theta$ is a generalized Carleson--Newman inner function if and only if the space 
$H^\infty/\Theta H^\infty$ satisfies the norm controlled inversion property: 
$$
\sup\bigl\{ \|1/f\|_{H^\infty/\Theta H^\infty}:f\in H^\infty,
\|f|\ZT\|_{\infty}\ge \varepsilon \bigr\}<\infty,\qquad 
\varepsilon>0.
$$
An equivalent property is formulated in terms of (the spectral properties of) the model operator 
$M_\Theta$ on $K_\Theta=H^2\ominus \Theta H^2$, 
$M_\Theta f=P_{K_\Theta}zf$.
Namely, for every $f\in H^\infty$, 
the operator $f(M_\Theta)$ is invertible if and only if the eigenvalues 
$f(\lambda)$, $\lambda\in \ZT$, of $f(M_\Theta)$ are bounded away from zero. 
For one more equivalent property we consider the space $\mathfrak M(H^\infty)$ of the maximal ideals of $H^\infty$. Denote by $Z_{\mathfrak M(H^\infty)}(\Theta)$ the zero set of $\Theta$ in $\mathfrak M(H^\infty)$.
Then $\Theta$ is a generalized Carleson--Newman inner function if and only if $Z_{\mathfrak M(H^\infty)}(\Theta)=\clos_{\mathfrak M(H^\infty)}\ZT$. 

Example~3.7 in \cite{GMN} shows that there 
exist generalized Carleson--New\-man Blaschke products  
which are not Carleson--Newman Blaschke products. For more information on generalized Carleson--Newman inner functions see \cite{GMN,GM,M,VN}. Still, not enough is known about the geometric characteristics of Blaschke sequences $\Lambda$ and about measures $\nu$ on $\T$ such that $B_\Lambda$, 
$B_\Lambda S_\nu$ are generalized Carleson--Newman inner functions. 

The aim of this note is to give more information about such inner functions. In particular, we answer (in the negative) Question~4 at the end of \cite{GMN}: there are inner functions $I$ that do not divide 
generalized Carleson--Newman inner functions.
Furthermore, we give new proofs for some facts about 
generalized Carleson--Newman inner functions. These proofs are more elementary than those in \cite{GM,M}: they do not use notions and results pertaining to the space of the maximal ideals of $H^\infty$.
\smallskip

{\bf Acknowledgements.}
The author is grateful to Raymond Mortini and Nikolai Nikolski for helpful discussions and comments.

'\section{Main results}

We say that a function $U\in H^\infty(\mathbb D)$ 
is a generalized Carleson--Newman function 
(or satisfies the $\WEP$ condition) if for some increasing function $\psi:(0,1)\to (0,1)$, 
\begin{equation}
|U(z)|\ge \psi(\rho(z,Z(U)),\qquad z\in\mathbb D.
\label{xda}
\end{equation}

We start with the following result proved in \cite[Theorem~3.6]{GM} using results concerning 
the space of the maximal ideals of $H^\infty$. 

\begin{proposition}\label{P1} 
Let $B$ be a Blaschke product with the zero set 
$\Lambda$,
\begin{equation}
|B(z)|\ge A\,\rho^n(z,\Lambda),\qquad z\in\mathbb D.
\label{1}
\end{equation}
Then $B$ is the product of $n$ interpolating Blaschke products.
\end{proposition}

In the next section we give a more elementary proof of this statement.

Furthermore, we formulate one more result established in 
\cite[Theorem~5.1]{GM} (they also cite D.~Suarez), \cite[Theorem~4.1]{M}
using the space of the maximal ideals of $H^\infty$.

Denote
$$
\varphi_w(z)=\frac{w-z}{1-z\bar w}.
$$

\begin{proposition}\label{P2} Let $U$ be a generalized Carleson--Newman inner function, and let
$\psi$ satisfy \eqref{xda}.
If $0<|\gamma|<\sup_{t<1}\psi(t)$, then
$U_\gamma=\varphi_\gamma(U)$ is a Carleson--Newman Blaschke product.
\end{proposition}

In the next section we give a more elementary proof of this statement.

\cite[Example~3.7]{GMN} gives a generalized Carleson--Newman Blaschke product 
$B$ which is not a Carleson--Newman Blaschke product.
Here we give a somewhat different example.

\begin{proposition}\label{P3} There exists a sequence of finite Blaschke products $(B_k)_{k\ge 1}$ such that if $|w_k|$ tend to $1$ sufficiently rapidly, then 
$$
B=\prod_{k\ge 1}B_k(\varphi_{w_k})
$$
is a generalized Carleson--Newman Blaschke product 
but not a Carle\-son--Newman Blaschke product.
\end{proposition}

This construction shows, in particular, that the indicator function 
$$
\eta_B(\varepsilon)=\sup\{\eta>0:|B(z)|<\eta\implies \rho(z,Z(B))<\varepsilon\}
$$
of a generalized Carleson--Newman Blaschke product may vanish arbitrarily rapidly: for every increasing $\psi:(0,1)\to(0,1)$ there exists 
a generalized Carleson--Newman Blaschke product $B$ such that 
$$
\eta_B(\varepsilon)=o(\psi(\varepsilon)),\qquad \varepsilon\to 0.
$$
On the other hand, one can use this construction to see that a finite pseudo-hyperbolic perturbation may make a Blaschke product to stop being 
generalized Carleson--Newman: there exists two Blaschke products $B,B^*$ 
with zero sets $\{z_k\},\{z^*_k\}$ correspondingly, such that $\sup_k\rho(z_k,z^*_k)<1$, 
$B\in\WEP$, $B^*\not\in\WEP$.

To deal with divisors of generalized Carleson--Newman inner functions, we introduce the following notion. We say that a function $U\in H^\infty$ satisfies the 
weighted area condition 
$(A)$ if for any $\varepsilon\in(0,1)$,
$$
\int_{\Delta(U,\varepsilon)}\frac{dm_2(z)}{1-|z|}=+\infty,
$$
where
$$
\Delta(U,\varepsilon)=\{z:|U(z)|<\varepsilon\}.
$$
\medskip

\begin{proposition}\label{L2} If $U\in(A)$, then there is no Blaschke product $B$ such that $BU$ is a generalized Carleson--Newman function.
\end{proposition}

The following result answers Question~4 in \cite{GMN}:

\begin{theorem}\label{T1} There exists a singular inner function 
$S=S_\nu\in(A)$. Therefore, there is no Blaschke product $B$ such that 
$BS$ is a generalized Carleson--Newman inner function.
\end{theorem}

An easy modification of the construction in the proof of Theorem~\ref{T1}
(using small Frostman shifts of singular factors of $S$), gives 
a Blaschke product $B\in(A)$. Therefore, there is no Blaschke product $B_1$ such that 
$B_1BS$ is a generalized Carleson--Newman inner function.
\smallskip

In the proof of Theorem~\ref{T1} we construct an atomic measure 
$\nu=\sum_k c_k\delta_{\lambda_k}$ such that $S_\nu$ cannot be multiplied into the $\WEP$ class. The set $\{\lambda_k\}$ is countable, and its 
entropy is infinite. A simple modification of the construction
in \cite[Example~3.7]{GMN} produces a measure $\nu_1=\sum_k \tilde c_k\delta_{\lambda_k}$ with $\tilde c_k>0$, and a Blaschke product 
$B_1$ such that $S_{\nu_1}B_1$ is a generalized Carleson--Newman inner function. 

Thus, just information on the support of $\mu$ is not enough to decide 
whether there exists a Blaschke product $B$ satisfying
$BS_\mu\in\WEP$. It is interesting whether one can construct
an example like that in Theorem~\ref{T1} such that the entropy of 
the support of $\nu$ is finite.
\smallskip

Proposition~\ref{L2} gives a necessary condition ($U\notin(A)$) for the 
existence of Blaschke products $B$ such that 
$BU\in\WEP$. Next we verify that this condition is not sufficient.
For simplicity, we pass here to the half-plane setting.

\begin{proposition}\label{E8} There exists a singular function $S$ in 
$\mathbb C_+=\{z:\Im z>0\}$ such that for any $\varepsilon\in(0,1)$,
$$
\mathcal A(\Delta(S,\varepsilon))=
\int_{\Delta(S,\varepsilon)}\frac{dm_2(z)}{\Im z}<\infty,
$$
but there is no Blaschke product $B$ such that $BS$ is a generalized Carle\-son--Newman inner function.
\end{proposition}

\section{Proofs}

\begin{proof}[Proof of Proposition~\ref{P1}:]  (a) 
Given $z\in\D$, $0<\delta<1$, denote
$$
D(z,\delta)=\{w\in\D:\rho(z,w)<\delta\}.
$$
Fix $z\in\D$ and let us verify that$$
M=\card [\Lambda\cap D(z,1/2)]\le N(A,n).
$$

Indeed, for some absolute constants $c_1,c_2>0$ we have 
\begin{gather*}
\max_{D(z,1/2)} |B|\le \exp(-c_1M),\\
\max_{w\in D(z,1/2)} \rho(w,\Lambda)\ge c_2/\sqrt M.
\end{gather*}
By \eqref{1}, we have
$$
A(c_2/\sqrt M)^n\le \exp(-c_1M),
$$
and, hence,
$$
M\le N(A,n).
$$
\medskip

(b) Let us verify that for some $\delta=\delta(A,n)>0$ we have 
$$
\card [\Lambda\cap D(z,\delta)]\le n,\qquad z\in\mathbb D.
$$

Otherwise, suppose that 
$$
\card [\Lambda\cap D(z,\delta)]>n 
$$
for some $z$ and sufficiently small $\delta$.
Then
$$
|B(w)|\le (4\delta)^{n+1},\qquad w\in\partial D(z,2\delta).
$$
Hence,
$$
\rho(w,\Lambda)\le A^{-1/n}(4\delta)^{(n+1)/n},\qquad w\in\partial D(z,2\delta).
$$
Therefore,
$$
\card [\Lambda\cap D(z,3\delta)]>c(A,n)\delta^{-1/n} ,
$$
that contradicts to (a) for sufficiently small $\delta$.
\medskip

(c) Next we consider the following graph. Its vertices are points of 
$\Lambda$, and two points $z$ and $w$ in $\Lambda$ 
are connected by an edge if and only if $\rho(z,w)<\delta/(2n)$,
with $\delta$ defined in (b). 
Now, the statement (b) implies that connected components
$\Lambda^\alpha\subset\Lambda$, $\alpha\in\mathcal A$, 
have at most $n$ elements. We divide $\Lambda$
into $n$ subsets $\Lambda_j$, $1\le j\le n$, in such a way that
$$
\card [\Lambda^\alpha\cap \Lambda_j]\le 1,\qquad \alpha\in\mathcal A,
\quad 1\le j\le n.
$$
Let us verify that $\Lambda_j$ are interpolating sets. 
Fix, say, $z\in\Lambda_1$. Since
$$
\card\bigl[\Lambda\cap D(z,\delta/(3n))\setminus D(z,\delta/(6n))\bigr]<n,
$$
we can choose 
$$
\delta_1\in\Bigl(\frac{\delta}{6n},\frac{\delta}{3n}\Bigr)
$$ 
such that
$$
\rho(\Lambda,\partial D(z,\delta_1))\ge \frac{\delta}{12n^2}.
$$ 
If $B_1$ is a Blaschke product constructed by 
$\Lambda_1$, then by \eqref{1}
$$
|B_1(w)|\ge |B(w)|\ge A\Bigl(\frac{\delta}{12n^2}\Bigr)^n,
\qquad w\in \partial D(z,\delta_1).
$$
Using the M\"obius transform $\varphi_z$, we conclude that 
$B_1$ satisfies the Carleson condition:
$$
|B^\prime_1(z)|(1-|z|^2)\ge \frac{3nA}{\delta}\Bigl(\frac{\delta}{12n^2}\Bigr)^n. 
$$
\end{proof}

\begin{proof}[Proof of Proposition~\ref{P2}:] To verify that the Blaschke factor of $U_\gamma$ is a Carleson--Newman Blaschke product it suffices to check that for some $c,k$ depending on 
$\gamma$, $a$ and $U$ we have 
\begin{equation}
|U_\gamma(w)|\ge c\rho^k(w,Z(U_\gamma)), \qquad w\in \mathbb D,
\label{2}
\end{equation}
and use Proposition~\ref{P1}.

Let $\varepsilon>0$, $0<a<1$, $(1+\varepsilon)|\gamma|<\psi(a)$.
Consider $w\in\mathbb D$ such that $|U_\gamma(w)|<\varepsilon/2$
(otherwise, there is nothing to check). Then $|U(w)|<\psi(a)$,
and 
$$
\rho(w,z_0)\le a<1
$$
for some $z_0\in Z(U)$. We have 
\begin{equation}
|U_\gamma(z_0)|=|\gamma|.
\label{x11}
\end{equation}
Comparing the behavior of the functions $(x-a)/(1-a x)$
and $x^A$ for $x\to 1-$ for large $A$, we find that for some 
$A=A(a,\gamma)<\infty$ and $c(a,\gamma)<1$,
the set
$$
\Omega(w,z_0)=\{z\in\mathbb D:\rho(z,w)< \rho^A(z,z_0)\}.
$$
is a subset of $D(z_0,c)$, and, hence, by \eqref{x11}, 
$$
\card(Z(U_\gamma)\cap\Omega(w,z_0))\le C=C(a,\gamma)<\infty.
$$
Let $U_\gamma=BS$, where $B$ is a Blaschke product, and $S$
is a singular function. 
We count the points in $Z(B)=Z(U_\gamma)$ taking into account their 
multiplicities in $B$. 
Then,
\begin{align*}
|B(w)|&=
\prod_{z\in Z(B)}\rho(w,z)\notag
\\&=\prod_{z\in Z(B)\cap \Omega(w,z_0)}\rho(w,z)
\cdot\prod_{z\in Z(B)\setminus \Omega(w,z_0)}\rho(w,z)\notag
\\ &\ge \rho^C(w,Z(U_\gamma))\cdot 
\prod_{z\in Z(B)\setminus \Omega(w,z_0)}\rho^A(z_0,z)\notag
\\&\ge \rho^C(w,Z(U_\gamma))\cdot \prod_{z\in Z(B)}\rho^A(z_0,z)
=\rho^C(w,Z(U_\gamma))\cdot |B(z_0)|^A\notag
\\
&\ge 
|\gamma|^A\rho^C(w,Z(U_\gamma)).
\label{3}
\end{align*}

Thus, by Proposition~\ref{P1}, $B$ is a Carleson--Newman Blaschke product,
that is a finite 
product of interpolating Blaschke products.

The function $\log(1/|S|)$ is harmonic and positive 
in the unit disc, and $\rho(z_0,w)\le a<1$. Therefore, by the Harnack inequality applied, say, to $\log(1/|S(\varphi_{z_0})|)$, we obtain
$$
\log\frac 1{|S(w)|}\le K(a)\log\frac 1{|S(z_0)|}
\le K(a)\log\frac 1{|\gamma|}.
$$
Hence, for sufficiently
small $\delta$, the set $\{z:|U_\gamma(z)|<\delta\}$
is contained in a neighborhood of $Z(U_\gamma)$
consisting of relatively compact components. This implies that $S=1$.
\end{proof}

\begin{proof}[Proof of Proposition~\ref{P3}:] Let us divide $\mathbb D$ into the union of   …Whitney squares 
$$
Q_{jk}=\{z:1-2^{-j+1}\le |z|< 1-2^{-j},\, 2\pi(k-1)2^{-j}
\le\arg z< 2\pi k2^{-j}\},
$$
$j\ge 1$, $1\le k\le 2^j$.
Given $N\ge 1$, consider a family $\Sigma_N$
of points in $\mathbb D$ such that
$$
\card (\Sigma_N\cap Q_{jk})=(N-j)^+=\max(N-j,0).
$$
Suppose that for $1\le j<N$, $1\le k\le 2^j$,
\begin{align*}
\sup_{w\in Q_{jk}}\rho(w,\Sigma_N)&\asymp (N-j)^{-1/2},\\
\rho(w,\Sigma_N\setminus\{w\})&\asymp (N-j)^{-1/2},\qquad
w\in Q_{jk}\cap \Sigma_N.
\end{align*}
Then for $j\ge 1$, $1\le k\le 2^j$, $w\in Q_{jk}$
we have
$$
\sum_{z\in\Sigma_N\setminus \{d_w\}}\log\frac1{\rho(z,w)}\asymp 
\bigl[(N-j)^+\bigr]^2+1.
$$
Therefore, if
$$
B_N(w)=\prod_{z\in\Sigma_N}\frac{w-z}{1-w\bar z},
$$
then
$$
|B_N(w)|\ge \psi(\rho(w,\Sigma_N)),\qquad w\in\mathbb D,
$$
where $\psi(x)=c_1x\exp(-c/x^4)$, with $c,c_1$ independent of $N$.

Now, if $|w_k|$ tend to $1$ sufficiently rapidly, then
$$
\prod_{k\ge 1}B_k(\varphi_{w_k})
$$
satisfies the $\WEP$ condition, but $B$ is not a Carleson--Newman Blaschke product.
\end{proof}

\begin{proof}[Proof of Proposition~\ref{L2}:] If $BU\in\WEP$, then for some 
$\varepsilon>0$ we have
$$
\Delta(U,\varepsilon)\subset D(Z(BU), 1/2),
$$
where $D(Z(BU), 1/2)$ is the $1/2$-pseudohyperbolic 
neighborhood of the set $Z(BU)$: 
$$
D(Z(BU), 1/2)=\cup_{w\in Z(BU)}D(w,1/2).
$$
Since
$$
\int_{D(w,1/2)}\frac{dm_2(z)}{1-|z|}\asymp 1-|w|,
$$
we obtain
$$
\int_{D(Z(BU), 1/2)}\frac{dm_2(z)}{1-|z|}\le C\cdot
\sum_{w\in Z(BU)}1-|w|<\infty. 
$$
\end{proof}

\begin{proof}[Proof of Theorem~\ref{T1}:]
Given $\varepsilon>0$, $1\le n\le N$, let  
$$
\mu=\mu_{\varepsilon,n,N}=\varepsilon\sum_{1\le k\le n}
\delta_{\exp(2\pi i k/N)}.
$$
Then $\|\mu\|=n\varepsilon$, and for some absolute constant $c>0$ we have 
$$
(\mathcal P*\mu)(z)\ge c\varepsilon N
$$
for $0<\arg z<2\pi n/N$, $1/N< 1-|z| < n/N$.
Therefore,
$$
\int_{z: (\mathcal P*\mu)(z)\ge c\varepsilon N}
\frac{dm_2(z)}{1-|z|}\ge C\int_0^{2\pi n/N}d\theta\int_{1/N}^{n/N}
\frac {dr}r\ge \frac{Cn\log n}{N}.
$$
Let us choose $\varepsilon_k,n_k,N_k$ satisfying the asymptotic properties
$$
\frac{n_k\log n_k}{N_k}\to\infty,\quad \varepsilon_kN_k\to\infty,
\quad k\to\infty,
\quad \sum_k \varepsilon_kn_k<\infty;
$$
for example, we can take 
$$
\varepsilon_k=2^{-k^2},\quad n_k=2^{k^2-k},\quad
N_k=k2^{k^2}.
$$
If $\nu$ is the sum of the measures
$\mu_{\varepsilon_k,n_k,N_k}$, $k\ge 1$, then $S_\nu\in(A)$.
Thus, by Lemma 2, there is no Blaschke product $B$ such that 
$BS$ satisfies the $\WEP$ condition.
\end{proof}

\begin{proof}[Proof of Proposition~\ref{E8}:]
It suffices to construct a singular function $S$ such that (i) 
$$
\mathcal A(\Delta(S,\delta))<\infty,\qquad 0<\delta<1,
$$
and (ii) if $\varepsilon>0$ and $B_1$ is a Blaschke product such that
$$
\Delta(S,\varepsilon)\subset D(Z(B_1), 1/2),
$$ 
then
$$
\mathcal A(\Delta(B_1,\varepsilon_1))=+\infty,\qquad 0<\varepsilon_1<1.
$$

For some $n_j, N_j$, $N_j> 2^{2n_j}$, set 
\begin{gather*}
\lambda_{m,s}=\frac{1}{N_j}\bigl(sn_j^{2/3}2^{n_j}+m\bigr),\\
\mu_j=\frac{n_j^{1/3}}{N_j}\sum_{0\le s<2^{n_j}}\sum_{0\le m< 2^{n_j}}
\delta_{\lambda_{m,s}}.
\end{gather*}
Let $u_j=\mathcal P*\mu_j$ be the Poisson integral of $\mu_j$ in the upper half-plane, and let $S$ be the singular function determined by the sum
of the measures $\mu_j$.

Then 
$$
\|\mu_j\|=
\frac{n_j^{1/3}}{N_j}\cdot 
2^{2n_j},
$$
and for $0<\delta<1$ we have 
$$
\mathcal A(\{z:u_j>\delta\}))\le c(\delta)\frac{n_j}{N_j}\cdot 2^{2n_j}.
$$

Furthermore, for some absolute constant $c>0$ we have 
$$
u_j\ge cn_j^{1/3}
$$
on the set
$$
\Omega_j=\bigcup_{0\le s<2^{n_j}}
\Bigl\{\frac1{N_j}\le \Im z\le \frac{2^{n_j}}{N_j},\,\,
0\le \Re z-\frac{n_j^{2/3}s2^{n_j}}{N_j}\le \frac{2^{n_j}}{N_j}\Bigr\}.
$$
Let $\Lambda_j$ be such that
$\Omega_j\subset D(\Lambda_j,1/2)$.  Then we have 
$$
\log\prod_{w\in\Lambda_j}\frac1{\rho(z,w)}\ge Cn_j^{1/3},\qquad z\in E_j, 
$$
where
$$
E_j=\Bigl\{\frac {n_j^{2/3}2^{n_j}}{N_j}\le \Im z\le 
\frac{n_j^{2/3}2^{2n_j}}{N_j},\,\,0\le \Re z\le 
\frac{n_j^{2/3}2^{2n_j}}{N_j}\Bigr\}.
$$
Finally,
$$
\mathcal A(E_j)\asymp\frac{n_j^{5/3}}{N_j}\cdot 2^{2n_j}.
$$
If now 
$$
\sum_{j\ge 1}\frac{n_j}{N_j}\cdot 2^{2n_j}<\infty,
$$
and
$$
\frac{n_j^{5/3}}{N_j}\cdot 2^{2n_j}\to\infty,\qquad j\to\infty,
$$
then our function $S$ satisfies all the above-mentioned conditions.
\end{proof}

\end{document}